\documentclass[12pt, reqno]{amsart}
\allowdisplaybreaks[1]
\usepackage{amsmath}
\usepackage{amssymb}
\usepackage{amsfonts}
\usepackage{verbatim}
\usepackage[usenames]{color}
\makeindex

 \newtheorem{theorem}{Theorem}[section]
 \newtheorem{corollary}[theorem]{Corollary}
 \newtheorem{lemma}[theorem]{Lemma}
 \newtheorem{proposition}[theorem]{Proposition}

\newtheorem{observation}[theorem]{Observation}

\theoremstyle{definition}
\newtheorem{definition}[theorem]{Definition}

\theoremstyle{remark}

\newtheorem{fact*}{Fact}

\newcommand{\hilbert}{\mathcal{H}}

\newcommand{\C}{\mathbb{C}}

\newcommand{\cc}[1]{\overline{#1}}

\newcommand{\norm}[1]{\left\Vert#1\right\Vert}

\newcommand{\tr}{\operatorname{tr}}
\newcommand{\ran}[1]{\operatorname{ran}#1}

\newcommand{\ip}[2]{\left\langle #1, #2 \right\rangle}
\newcommand{\ad}{^\ast}

\newcommand{\til}{\raise.17ex\hbox{$\scriptstyle\mathtt{\sim}$}}

\newcommand\la{\lambda}

\newcommand\beq{\begin{equation}}

\newcommand\eeq{\end{equation}}

\newcommand\black{\color{black}}

\newcommand\bbm{\begin{bmatrix}}
\newcommand\ebm{\end{bmatrix}}
\newcommand{\bpm}{\left( \begin{smallmatrix}}
\newcommand{\epm}{\end{smallmatrix} \right)}
\numberwithin{equation}{section}

\newcommand{\tensor}[2]{\text{ }{\begin{smallmatrix} #1 \\ \otimes\\ #2\end{smallmatrix}}\text{  }}
\newcommand{\flattensor}[2]{#1 \otimes #2}

\newlength{\Mheight}
\newlength{\cwidth}
\newcommand{\mc}{\settoheight{\Mheight}{M}\settowidth{\cwidth}{c}M\parbox[b][\Mheight][t]{\cwidth}{c}}

\newcommand{\dfn}[1]{{\bf #1}\index{#1}}

\title[Free functions with symmetry]{Free functions with symmetry}
\author[
David Cushing
	J. E. Pascoe
	Ryan Tully-Doyle	
]{
	David Cushing$^\dagger$ \\
	J. E. Pascoe$^{\ddagger}$ \\
	Ryan Tully-Doyle \\
}
\thanks{\\
$\dagger$ Partially supported by an EPSRC DTA grant and a Capstaff award.\\
$\ddagger$ Partially supported by National Science Foundation Grant DMS 1361720}
\date{\today}

 	

\begin{document}

\begin{abstract}
In 1936, Margarete C. Wolf showed that the ring of symmetric free polynomials in two or more variables is
isomorphic to the ring of free polynomials in infinitely many variables.
We show that Wolf's theorem is a special case of a general theory of the ring of invariant
free polynomials: every ring of invariant free polynomials is isomorphic to
a free polynomial ring. Furthermore, we show that this isomorphism
extends to the free functional calculus as a norm-preserving isomorphism of function spaces
on a domain known as the row ball.
We give explicit constructions of the ring of invariant free polynomials in terms of representation theory and
develop a rudimentary theory of their structures. Specifically, we obtain
a generating function for the number of basis elements of a given degree
and explicit formulas for good bases in the abelian case.
\end{abstract}
\maketitle

\tableofcontents
\section{Introduction}



A \dfn{symmetric free polynomial} in the letters $x_1, \ldots, x_d$ is a free polynomial
$p$ which satisfies
	$$p(x_1,\ldots, x_d) = p(x_{\sigma(1)},\ldots , x_{\sigma(d)})$$
for all permutations $\sigma$ of $\{1,\ldots, d\}.$ For example,
	$$p(x_1,x_2) = x_1 + x_2 +7x_1x_1 + 7x_2x_2 + 3 x_1x_2 + 3x_2x_1$$
is a symmetric free polynomial in two variables.
Classically, Wolf proved the following theorem about the structure of
the ring of symmetric free polynomials \cite{wolf36}.
\begin{theorem}[Wolf]
The symmetric free polynomials in $d \geq 2$ variables is isomorphic to the ring of free polynomials in infinitely many variables.
\end{theorem}
More recently, free symmetric polynomials have been investigated by numerous authors \cite{dia12, brrz08, gel95}, and some deep connections with representation theory are now known.

We are concerned with the ring of invariant free polynomials.
Let $G$ be a finite group.  Let $\pi: G \to \mathcal U_d$ be a unitary representation.
That is,
$\pi$ is a homomorphism from the group $G$ to the group of
$d \times d$ unitary matrices over $\mathbb{C}.$  
An \dfn{invariant free polynomial with respect to $\pi$} is a free polynomial $p \in \C\langle x_1, \ldots, x_d\rangle$ which satisfies
	$$p(x_1,\ldots, x_d) = p(\pi(\sigma)(x_1,\ldots, x_d))$$
for all $\sigma$ in $G.$
(Here, the $d \times d$ matrices $\pi(\sigma)$ are acting on $(x_1,\ldots, x_d)$
by matrix multiplication. That is, for example, if
$$A = \bpm
1 & 2 \\
3 & 4
\epm,$$
we define
$A(x_1, x_2) = (x_1+2x_2, 3x_1 + 4x_2).$) Notably, the invariant free polynomials
form an algebra, which is called the \dfn{ring of invariant free polynomials.}

For example, the symmetric free polynomials in $d$ variables are obtained
in the special case where $G = S_d$ a symmetric group, and we define $\pi: S_d \to \mathcal U_d$ be the representation satisfying
	$$\pi(\sigma)e_i = e_{\sigma(i)},$$
where $e_i$ is the $i$-th elementary basis vector for $\mathbb{C}^d.$

We now define the domains on which we intend to execute the function theory
of invariant free polynomials.
Fix an infinite dimensional separable Hilbert space $\hilbert$. There is
only one infinite dimensional separable Hilbert space up to isomorphism \cite{con85},
so what
follows will be independent of the exact choice.
Let $\Lambda$ be an index set.
Let $\mathcal C^\Lambda$ be the set of $(X_\lambda)_{\lambda \in \Lambda},$
sequences of elements in $\mathcal{B}(\hilbert)$ indexed by $\Lambda,$ such that 
$$\sum_{\lambda \in \Lambda} X_\lambda X_\lambda\ad < 1,$$ 
where $A < B$ means that $B-A$ is strictly positive definite and $A \leq B$ means that
$B-A$ is positive semidefinite.
Here the sum $\sum_{\lambda \in \Lambda} X_\lambda X_\lambda\ad$ is required to be absolutely convergent in the weak operator topology and have $\sum_{\lambda \in \Lambda} X_\lambda X_\lambda\ad < 1,$ or equivalently, there is an $\varepsilon>0$ such that 
for every finite $\Lambda'\subset \Lambda,$
	$$\sum_{\lambda \in \Lambda'} X_\lambda X_\lambda\ad \leq 1-\varepsilon.$$
If $\Lambda = d$ is a natural number, we identify $d$ as a set with $d$ elements.
That is, $\mathcal C^d$ is the set of $d$-tuples of operators in $\mathcal{B}(\hilbert)$
such that
	$$\sum^d_{i=1} X_i X_i\ad < 1.$$
Some authors refer to $\mathcal C^\Lambda$ as the \dfn{row ball}, or the set of \dfn{row contractions} \cite{hkms09,po06}.
Given a free polynomial $p$
in the letters $(x_\lambda)_{\lambda \in \Lambda}$
and
a point $X = (X_\lambda)_{\lambda\in \Lambda} \in \mathcal C^\Lambda$,
we form $p(X)$ via the formula $$p(X) = p((X_\lambda)_{\lambda\in \Lambda}),$$
which when $\Lambda = d$ reduces to $p(X) = p(X_1, \ldots, X_d).$

Let $\mathcal R$ be a subalgebra of $\C\langle x_1, \ldots, x_d\rangle.$ We define a \dfn{basis}
for $\mathcal R$ to be an indexed sequence $(u_{\lambda})_{\lambda\in\Lambda}$ of free polynomials which generate $\mathcal R$ as an algebra, which is minimal in the sense that
there is no $\Lambda ' \subsetneq \Lambda$  such that the sequence $(u_{\lambda})_{\lambda\in\Lambda '}$ generate $\mathcal R$ as an algebra.
We note that any basis
for $\mathcal R$ is \emph{countable}, since $\mathcal R$ has countable dimension as a vector
space as it is a subspace of $\C\langle x_1, \ldots, x_d\rangle,$ and the minimality condition implies that the $(u_{\lambda})_{\lambda\in\Lambda}$
are linearly independent.

We prove the following theorem. 
\begin{theorem}\label{mainresultpolynomials} \label{mainresult}
Let $G$ be a finite group.
Let $\pi:G\rightarrow\mathcal{U}_{d}$ be a unitary representation. 
There exists a basis for the ring of invariant free polynomials 
$(u_{\lambda})_{\lambda\in\Lambda}$ such that the map $\Phi$ on $\mathcal{C}^d$
defined by the formula
$$\Phi(X)=(u_{\lambda}(X))_{\lambda\in\Lambda}$$ satisfies the following properties:
\begin{itemize}
\item
The map $\Phi$
takes $\mathcal{C}^{d}$ to $\mathcal{C}^{\Lambda}.$
\item
Furthermore, for $p$ in the ring of invariant free polynomials for $\pi$, there exists a 
unique free polynomial $\hat{p}$ such that $p=\hat{p}\circ\Phi.$
\item Moreover,
$$\sup_{X \in \mathcal{C}^{d}} \|p(X)\| =
	\sup_{U \in \mathcal{C}^{\Lambda}} \|\hat{p}(U)\|.$$
\end{itemize}
Namely, the map taking $\hat{p}$
to $p$ is an isomorphism of rings from the free algebra
$\mathbb{C}\langle x_{\lambda}\rangle_{\lambda \in \Lambda}$
to the ring of invariant free polynomials.
\end{theorem}
Theorem \ref{mainresult} follows from Theorems \ref{mainresultso} and \ref{freeringortho}.
We note that a similar result was obtained for free functions on the domain
	$$\mathcal B^2 = \{ (X_1, X_2) \in \mathcal{B}(\hilbert) | \norm{X_i} < 1\}$$
for symmetric free functions in two variables by Agler and Young \cite{ay14}.

In general, the basis in Theorem \ref{mainresult} is hard to compute. However, the number of elements of a certain degree is computed in Section \ref{counting}. An explicit basis can be obtained if $G$ is abelian, which we give in Section \ref{abelian}.

\subsection{Some examples}

We now give some concrete examples of what our main result says.
First, we give an analogous theorem to that obtained in Agler and Young 
for symmetric free functions in two variables \cite{ay14}.
\begin{proposition}\label{sym2propexam}
Let
$$\Phi(X_1,X_2)=(A,B^2,BAB,\ldots, BA^nB,\ldots)$$
where $$A = \frac{X_1 + X_2}{\sqrt 2},
	B = \frac{X_1 - X_2}{\sqrt 2}.$$
The map $\Phi$ satisfies the following properties:
\begin{itemize}
	\item $\Phi$ takes $\mathcal C^2$ to $\mathcal C^\mathbb{N}.$
	\item For any free polynomial $p$ such that
		$$p(X_1,X_2) = p(X_2,X_1),$$
	there exists a
	unique free polynomial $\hat{p}$ such that $p=\hat{p}\circ\Phi.$
	\item Moreover,
		$$\sup_{X \in \mathcal{C}^{2}} \|p(X)\| =
	\sup_{U \in \mathcal{C}^{\mathbb{N}}} \|\hat{p}(U)\|.$$
\end{itemize}
\end{proposition}

Another simple example concerns the ring of even functions in two variables, that is, the ring of free polynomials in two variables $f$ satisfying the 
identity 
	$$f(X_1,X_2)= f(-X_1,-X_2).$$
Here the group $G$ in Theorem \ref{mainresult} is the cyclic group with two
elements, $\mathbb{Z}_2 = \{0,1\},$ and the representation $\pi$ is given by
	$$\pi(0) = \bpm 1 & 0 \\ 0 & 1 \epm, \pi(1) = \bpm -1 & 0 \\ 0 & -1 \epm.$$
\begin{proposition}\label{evenpropexam}
Let
$$\Phi(X_1,X_2)=(X_1X_1,X_1X_2,X_2X_1,X_2X_2).$$
The map $\Phi$ satisfies the following properties:
\begin{itemize}
	\item $\Phi$ takes $\mathcal C^2$ to $\mathcal C^4.$
	\item For any free polynomial $p$ such that
		$$p(X_1,X_2) = p(-X_1,-X_2),$$
	there exists a
	unique free polynomial $\hat{p}$ such that $p=\hat{p}\circ\Phi.$
	\item Moreover,
		$$\sup_{X \in \mathcal{C}^{2}} \|p(X)\| =
	\sup_{U \in \mathcal{C}^{4}} \|\hat{p}(U)\|.$$
\end{itemize}
\end{proposition}
We discuss Proposition \ref{evenpropexam} in detail in Section \ref{evensection}.
Here we see a concrete trade-off between the number of variables and degree in the
optimization of a free polynomial: finding the maximum norm of a polynomial in $4$ variables
of degree $d$ on $\mathcal C^4$
is the same as finding the norm of an even polynomial in $2$ variables of degree $2d$
on $\mathcal C^2.$

Both Proposition \ref{sym2propexam} and Proposition \ref{evenpropexam}
follow from our basis construction in the abelian case
given in Theorem \ref{abelianbasistheorem}.

\subsection{Geometry}
Although the image of the map $\Phi$ in Theorem \ref{mainresult} may have high codimension, in the sense that it is far from being literally surjective, the function $\hat{f}$ is completely determined by $f$ and has the same norm. We view this as an analogue of the celebrated work of Agler and {\mc}Carthy on norm preserving extensions of functions on varieties to whole domains \cite{agmcdv, agmcvn}. Additionally, since $f$ totally determines $\hat{f}$, the dimension of the Zariski closure of the image of a free polynomial map can apparently go up, in contrast with the commutative case \cite{hart97}.
Exploiting the aforementioned phenomenon is a critical step in the theory of
change of variables for free polynomials and their generalizations, the free functions, which had been thought to be extremely
rigid \cite{helkm11}.

To prove Theorem \ref{mainresult}, we develop a geometric theory of the ring of invariant free functions. As a consequence of the geometric structure, the ring of free invariant polynomials is always free (but perhaps infinitely generated), in contrast with the Chevalley-Shepard-Todd Theorem in the commutative case, in which freeness depends on the structure of $G$ \cite{chev55,st54}.

\subsection{Free analysis}
A \dfn{free function on $\mathcal{C}^{\Lambda}$}
is a function from $\mathcal{C}^{\Lambda}$ to $\mathcal{B}(\mathcal H)$
which is the uniform limit of free polynomials on the sets $r\cc{\mathcal{C}^{\Lambda}}.$
We let $H(\mathcal{C}^{\Lambda})$ denote the \dfn{algebra of
free functions on $\mathcal C^\Lambda$}.
The \dfn{Banach algebra of bounded free functions} $H^\infty(\mathcal{C}^{\Lambda}) \subset H(\mathcal{C}^{\Lambda})$ is
the space of all bounded functions in $H(\mathcal{C}^{\Lambda})$
equipped with the norm
	$$\|p\| = \sup_{X \in \mathcal{C}^{\Lambda}} \|p(X)\|.$$
We note that there are many equivalent characterizations of a free function, such as \cite{agmc_gh, hkms09,  vvw12,po06}.

We can reinterpret Theorem \ref{mainresult} as an isomorphism of function algebras analogous to Wolf's theorem.
Let $H^\infty_{\pi}(\mathcal C^d)$ be the Banach algebra of bounded invariant free functions for $\pi$ on $\mathcal C^d$.
\begin{corollary}\label{corollaryIso}
There is a countable set $\Lambda$ such that
	$$H^\infty_\pi(\mathcal C^d) \cong H^\infty(\mathcal C^\Lambda)$$
as Banach algebras.
\end{corollary}
Corollary \ref{corollaryIso} follows from Theorem \ref{mainresult} directly.
That is, the map taking
$\hat{f} \in H^\infty(\mathcal C^\Lambda)$ to
$\hat{f}\circ \Phi \in H^\infty_\pi(\mathcal C^d)$ is such an isomorphism.
The map is injective and surjective since it is already since
it is an isomorphism on the level of formal power series,
which uniquely define a free function \cite{vvw12}.

\section{Example: An even free function in two variables}\label{evensection}
	To begin, we discuss a simple nontrivial example of a free ring of invariants.
	We will now explain what our main result, Theorem \ref{mainresult}, says about a specific even free polynomial
	in two variables.

	We say a free polynomial $p \in \mathbb{C}\langle x_1,x_2\rangle$
	is \dfn{even} if
		$$p(X_1,X_2) = p(-X_1, -X_2).$$
	We note that the even free functions form an algebra.
	
	Consider the even free polynomial
		$$p(X) = p(X_1,X_2) = 1 + 3X_1X_2 - 7X_1X_1 - X_2X_1X_2X_2$$
	as a map on the domain $\mathcal{C}^2.$
	
	We first note that it is clearly not a coincidence that $p$ has no odd degree terms.
	Furthermore, if we let
		$$\begin{matrix}
		u_{1}(X) = X_1X_1, & u_{2}(X) = X_1X_2, \\
		u_{3}(X) = X_2X_1, & u_{4}(X) = X_2X_2,
		\end{matrix}$$
	we get that 
		$$p(X) = 1 + 3u_{2}(X) - 7u_{1}(X) - u_{3}(X)u_{4}(X).$$
	Let
		$$\hat{p}(U) = \hat{p}(U_1,U_2,U_3,U_4) = 1 + 3U_{2} - 7U_{1} - U_{3}U_{4}$$
	and
		$$\Phi(X) = (X_1X_1, X_1X_2, X_2X_1, X_2X_2).$$
	Thus,  $$p(X) = \hat{p} \circ \Phi(X).$$
	We are interested in the analytical properties of $\hat{p}$ and $\Phi.$ 
	
	We will show the remarkable fact that $$\Phi(\mathcal{C}^2) \subset \mathcal{C}^4.$$
	Let $X \in \mathcal{C}^2.$ That is,
		$$XX^* = X_1X_1^* + X_2X_2^* < 1.$$
	Since
		$$\Phi(X) = (X_1X_1, X_1X_2, X_2X_1, X_2X_2)$$
	we get that 
		\begin{align*}
			\Phi(X)\Phi(X)^* & = X_1X_1(X_1X_1)^* + X_1X_2(X_1X_2)^* + X_2X_1(X_2X_1)^* \\
			& \phantom{=} \hspace{.5in} + X_2X_2(X_2X_2)^* \\
			& =  X_1X_1X_1^*X_1^* + X_1X_2X_2^*X_1^* + X_2X_1X_1^*X_2^* +X_2X_2X_2^*X_2^* \\
			& =  X_1(X_1X_1^* + X_2X_2^*)X_1^* + X_2(X_1X_1^* + X_2X_2^*)X_2^* \\
			& \leq  X_1X_1^* + X_2X_2^* \\
			& <  1. \\
		\end{align*} 
	Thus, $\Phi(\mathcal{C}^2)  \subset \mathcal{C}^4.$
	

		
	We will now show a curious equality:
		$$\sup_{X \in \mathcal{C}^2} \|p(X)\| = \sup_{U \in \mathcal{C}^4} \|\hat{p}(U)\|.$$
	
	First we show that
		$$\sup_{X \in \mathcal{C}^2} \|p(X)\| \geq \sup_{U \in \mathcal{C}^4} \|\hat{p}(U)\|.$$
	Since $p$
	is a free polynomial and is thus continuous on the closure
	$\cc{\mathcal{C}^2} \subset \mathcal{B}(\hilbert)^2,$
	it is enough to show that
	$$\sup_{X \in \cc{\mathcal{C}^2}} \|p(X)\| \geq \sup_{U \in \mathcal{C}^4} \|\hat{p}(U)\|.$$
	
	Let $(U_1,U_2,U_3,U_4) \in \mathcal{C}^4.$
	Define operators $X_1$ and $X_2$ with the following block structure
		$$X_1 = \bpm
			0 & U_1 & U_2 \\
			1 & 0 & 0 \\
			0 & 0 & 0\epm,$$
		$$X_2 = \bpm
			0 & U_3 & U_4 \\
			0 & 0 & 0 \\
			1 & 0 & 0\epm.$$
	First,
		\begin{align*} X_1X_1^* + X_2X_2^*
		& = 	\bpm
			0 & U_1 & U_2 \\
			1 & 0 & 0 \\
			0 & 0 & 0\epm
			\bpm
			0 & U_1 & U_2 \\
			1 & 0 & 0 \\
			0 & 0 & 0\epm^*
			+
			\bpm
			0 & U_3 & U_4 \\
			0 & 0 & 0 \\
			1 & 0 & 0\epm
			\bpm
			0 & U_3 & U_4 \\
			0 & 0 & 0 \\
			1 & 0 & 0\epm^* \\ \\
		& = 	\bpm
			0 & U_1 & U_2 \\
			1 & 0 & 0 \\
			0 & 0 & 0\epm
			\bpm
			0 & 1 & 0 \\
			U_1^* & 0 & 0 \\
			U_2^* & 0 & 0\epm
			+
			\bpm
			0 & U_3 & U_4 \\
			0 & 0 & 0 \\
			1 & 0 & 0\epm
			\bpm
			0 & 0 & 1 \\
			U_3^* & 0 & 0 \\
			U_4^* & 0 & 0\epm \\ \\
		& = 	\bpm
			U_1U_1^* + U_2U_2^* & 0 & 0 \\
			0 & 1 & 0 \\
			0 & 0 & 0\epm
			+
			\bpm
			U_3U_3^* + U_4U_4^* & 0 & 0 \\
			0 & 0 & 0 \\
			0 & 0 & 1\epm \\ \\
		& = 	\bpm
			U_1U_1^* + U_2U_2^* + U_3U_3^* + U_4U_4^* & 0 & 0 \\
			0 & 1 & 0 \\
			0 & 0 & 1\epm \\
		 &\leq 1.\end{align*}
	That is, $(X_1, X_2) \in \cc{\mathcal{C}^2}.$
	Note 
		\begin{align*}
		X_1X_1 = \bpm
			0 & U_1 & U_2 \\
			1 & 0 & 0 \\
			0 & 0 & 0\epm
			\bpm
			0 & U_1 & U_2 \\
			1 & 0 & 0 \\
			0 & 0 & 0\epm
		=
			\bpm
			U_1 & 0 & 0 \\
			0 & U_1 & U_2 \\
			0 & 0 & 0\epm
		\\ \\
		X_1X_2 = \bpm
			0 & U_1 & U_2 \\
			1 & 0 & 0 \\
			0 & 0 & 0\epm
			\bpm
			0 & U_3 & U_4 \\
			0 & 0 & 0 \\
			1 & 0 & 0\epm
		=
			\bpm
			U_2 & 0 & 0 \\
			0 & U_3 & U_4 \\
			0 & 0 & 0\epm
		\\ \\	
		X_2X_1 =
			\bpm
			0 & U_3 & U_4 \\
			0 & 0 & 0 \\
			1 & 0 & 0\epm
			\bpm
			0 & U_1 & U_2 \\
			1 & 0 & 0 \\
			0 & 0 & 0\epm
		=
			\bpm
			U_3 & 0 & 0 \\
			0 & 0 & 0 \\
			0 & U_1 & U_2\epm
		\\ \\	
		X_2X_2 =
			\bpm
			0 & U_3 & U_4 \\
			0 & 0 & 0 \\
			1 & 0 & 0\epm
			\bpm
			0 & U_3 & U_4 \\
			0 & 0 & 0 \\
			1 & 0 & 0\epm
		=
			\bpm
			U_4 & 0 & 0 \\
			0 & 0 & 0 \\
			0 & U_3 & U_4\epm.
		\end{align*}
	That is,
		$$\Phi(X) =  \bpm
			U & 0 \\
			0 & \tilde{J} \epm$$
	where $\tilde{J}$ denotes some matrix tuple which is irrelevant to our aims.
	So,
		\begin{align*}
		p(X_1,X_2) & = (\hat{p} \circ \Phi)(X_1,X_2) \\
		&=\hat{p}(X_1X_1,X_1X_2,X_2X_1,X_2X_2) \\
		& = \bpm
			\hat{p}(U) 	& 0 \\
			0 		& \hat{p}(\tilde{J})\epm.
		\end{align*}
	Namely,
		$$ \|p(X)\| \geq \|\hat{p}(U)\|.$$
	Thus,
		$$\sup_{X \in \cc{\mathcal{C}^2}} \|p(X)\| \geq \sup_{U \in \mathcal{C}^4} \|\hat{p}(U)\|.$$
	
	To see that
		$$\sup_{X \in \mathcal{C}^2} \|p(X)\| \leq \sup_{U \in \mathcal{C}^4} \|\hat{p}(U)\|,$$
	note that since $\Phi(X) \in \mathcal{C}^4,$
		\begin{align*}
			\sup_{X \in \mathcal{C}^2} \|p(X)\| & = \sup_{X \in \mathcal{C}^2} \|\hat{p}(\Phi(X))\| \\
			& \leq \sup_{U \in \mathcal{C}^4} \|\hat{p}(U)\|.
		\end{align*}
	


\section{The space of free polynomials as an inner product space}
	We now seek to understand the geometry of free polynomials as a Hilbert space.
	\begin{definition}\label{hardydef}
	We define
	$H^2_d$ to be the Hilbert space of free formal power series in $d$ variables whose
	coefficients are in $\ell^2.$ We define $[H^2_d]_n$ to be the subspace of homogeneous
	free polynomials of degree $n.$ For a free polynomial $f$ in $d$ variables,
	we define $M_f$ to be the operator on $H^2_d$ satisfying $M_f g = fg.$
	\end{definition}

	The following lemma describes a grading structure on $H^2_d.$
	\begin{lemma}[Grading lemma]
		Let $p, q$ be homogenous free polynomials of degree $n$
		and $r, s$ be homogenous free polynomials of degree $m.$
		Then,
			$$\ip{pr}{qs} = \ip{p}{q}\ip{r}{s}.$$
		Thus, we identify  $\flattensor{[H^2_d]_n}{[H^2_d]_m} = [H^2_d]_{n+m}.$
	\end{lemma}
	\begin{proof}
		Write
		$$p(X)=\sum_{I\in \mathcal{I}_{n}}p_{I}X^{I},q(X)=\sum_{I\in \mathcal{I}_{n}}q_{I}X^{I},$$ 
		$$r(X)=\sum_{J\in \mathcal{I}_{m}}r_{J}X^{J}, s(X)=\sum_{J\in \mathcal{I}_{m}}s_{J}X^{J},$$ 
		where $\mathcal{I}_{k}$ is the set of free multi-indicies of degree $k.$
		Observe that $\ip{p}{q}=\sum_{I\in \mathcal{I}_{n}}p_{I}\overline{q_{I}}$
		and $\ip{r}{s}=\sum_{J\in \mathcal{I}_{m}}r_{J}\overline{s_{J}}$, which gives 
		$$\ip{p}{q}\ip{r}{s}=\sum_{I\in \mathcal{I}_{n}}
		\sum_{J\in \mathcal{I}_{m}}
		p_{I}\overline{q_{I}}r_{J}\overline{s_{J}}.$$
		 
		Observe that
		$$p(X)r(X)=\sum_{I\in \mathcal{I}_{n}}
		\sum_{J\in \mathcal{I}_{m}}
		p_{I}r_{J}X^{IJ},$$
		$$q(X)s(X)=
		\sum_{I\in \mathcal{I}_{n}}
		\sum_{J\in \mathcal{I}_{m}}
		q_{I}s_{J}X^{IJ},$$ which gives
		$$\ip{pr}{qs}=\sum_{I\in \mathcal{I}_{n}}
		\sum_{J\in \mathcal{I}_{m}}p_{I}r_{J}\overline{q_I}\overline{s_{J}}=\ip{p}{q}\ip{r}{s}.$$
	\end{proof}

	\begin{corollary}\label{orthosumsame}
		Let $V$ be a subspace of $[H^2_d]_n.$
		Let $u_1, \ldots, u_k$ be an orthonormal basis for $V$
		and $v_1, \ldots, v_k$ be an orthonormal basis for $V.$
		Then,
			$$\sum u_k(X) u_k(X)^* = \sum v_k(X)v_k(X)^*.$$
	\end{corollary}
	\begin{proof}
		Let $m_X = (X^I)_{|I|=n}$ be the row vector of free monomials of degree $n.$
		Note that, for any $p \in [H^2_d]_n,$
			$$p(X) p(X)^* = m_X\tensor{pp^*}{1}m_X.$$
			
		For example, if we took the polynomial $p = x_1 + 2x_2,$
		identifying $[H^2_2]_1$ with $\mathbb{C}^2$
		we get that
			$$p(X)p(X)^* =
			\bpm X_1 & X_2 \epm \bpm 1 \\ 2 \epm\bpm 1 & 2 \epm\bpm X_1^* \\ X_2^* \epm.$$
		
		So we get that
			$$\sum u_i(X) u_i(X)^*
			= \sum m_X\tensor{u_iu_i^*}{1}m_X
			= m_X \tensor{\sum u_iu_i^*}{1}m_X.$$	
		Note that $\sum u_iu_i^*$ is the orthogonal projection onto $V,$
		and thus did not depend on the choice of basis, so we are done.
	\end{proof}
	

	\section{Superorthogonality}
	
	We can now define superorthogonality.
	We will later see that the ring of invariant free polynomials
	is itself superorthogonal.
	\begin{definition}
	An algebra $\mathcal A \subseteq \C\langle x_1, \ldots, x_d \rangle$ is called
	\dfn{superorthogonal} if there exists a superorthonormal basis for $\mathcal A$,
	i.e. a basis $(u_\la)_{\la \in \Lambda}$ for $\mathcal A$ such that $u_\la$
	are homogeneous polynomials, $\norm{u_\la}_{\hilbert^2_d} = 1$, and for all $\la, \mu \in \Lambda$,
		$$u_\la H^2_d \perp u_\mu H^2_d.$$
	
	\end{definition}
	The goal of this section is to prove that a general version of Theorem \ref{mainresult}
	is true for superorthogonal algebras.	
	\begin{theorem}\label{mainresultso}
	Let $\mathcal{A}$ be a superorthogonal algebra.
	Let
	$(u_{\lambda})_{\lambda\in\Lambda}$ be a superorthogonal basis.
	The map $\Phi$ on $\mathcal{C}^d$
	defined by the formula
	$$\Phi(X)=(u_{\lambda}(X))_{\lambda\in\Lambda}$$ satisfies the following properties:
	\begin{itemize}
	\item
	The map $\Phi$
	takes $\mathcal{C}^{d}$ to $\mathcal{C}^{\Lambda}.$
	\item
	Furthermore, for $p$ in the ring of invariant free polynomials for $\pi$, there exists a 
	unique free polynomial $\hat{p}$ such that $p=\hat{p}\circ\Phi.$
	\item Moreover,
	$$\sup_{X \in \mathcal{C}^{d}} \|p(X)\| =
		\sup_{U \in \mathcal{C}^{\Lambda}} \|\hat{p}(U)\|.$$
	\end{itemize}
	Namely, the map taking $\hat{p}$
	to $p$ is an isomorphism of rings from the free algebra
	$\mathbb{C}\langle x_{\lambda}\rangle_{\lambda \in \Lambda}$
	to $\mathcal{A}$.
	\end{theorem}
	Theorem \ref{mainresultso} follows from Lemma \ref{sofree}
	and Lemma \ref{ball}.
	
	We now show that superorthogonal algebras are necessarily free. 
	\begin{lemma}\label{sofree}
		A superorthogonal algebra $\mathcal{A}$ is isomorphic to a free algebra in perhaps infinitely many variables. Specifically, the map
		$$\varphi: \mathbb{C}\langle x_\la \rangle_{\la \in \Lambda}\rightarrow \mathcal A$$
		satisfying $\varphi(x_{\lambda}) = u_{\lambda}$ is an isomorphism.
	\end{lemma}
	\begin{proof}
		Note that it is enough to show that for
		any two distinct products
		$\prod^{n}_{i=1} u_{\lambda_i},$ $\prod^{m}_{j=1} u_{\kappa_j}$		
		of the same degree as free polynomials in
			$\mathcal A \subset \C\langle x_1, \ldots, x_d \rangle$
		that
			$$\ip{\prod^{n}_{i=1} u_{\lambda_i}}{\prod^{m}_{j=1} u_{\kappa_j}}=0,$$
		since then the words in $u_\lambda$ are linearly independent.
		Since the words are distinct,
		there is a $p$ such that $\lambda_i = \kappa_i$ for all $i<p$
		and $\lambda_p \neq \kappa_p.$
		So,
		$$\ip{\prod^{n}_{i=1} u_{\lambda_i}}{\prod^{m}_{j=1} u_{\kappa_j}}=
		\ip{\prod^{p-1}_{i=1} u_{\lambda_i}}{\prod^{p-1}_{j=1} u_{\kappa_j}}
		\ip{\prod^{n}_{i=p} u_{\lambda_i}}{\prod^{m}_{j=p} u_{\kappa_j}}.$$
		Note that $\prod^{n}_{i=p} u_{\lambda_i} \in u_{\lambda_p}H^2_d$ and
		$\prod^{m}_{j=p} u_{\kappa_j} \in u_{\kappa_p}H^2_d$
		which implies that the two products are orthogonal since 
		$u_{\lambda_p}$ and $u_{\kappa_p}$ are superorthogonal and thus the desired inner
		product is $0.$
	\end{proof}
	
	We now show that the superothogonal basis maps $\mathcal C^{d}$
	into $\mathcal C^{\Lambda}$.
	\begin{lemma}\label{basismap}
	Let $\mathcal{A}$ be a superorthogonal algebra.
	For a superorthonormal basis $(u_\la)$, the map
	$\Phi:\mathcal{C}^d \to \mathcal{B}(\hilbert)^{\Lambda}$ given by
		$$\Phi(x) = (u_\la(x))_{\la \in \Lambda}$$
	has $\ran\Phi \subseteq \mathcal C^{\Lambda}$, that is
		$$\Phi: \mathcal{C}^d \to \mathcal{C}^\Lambda.$$
	We call $\Phi$ the \dfn{superorthogonal basis map}.
	\end{lemma}
	\begin{proof}
	Let $X \in \mathcal C^d.$
	That is,
		$$\sum^{d}_{i=1} X_iX_i^* <1.$$
	Let $\mathcal I$ be the left ideal generated by $\mathcal A \setminus \{1\}$ in
	$\mathbb{C}\langle x_1, \ldots, x_d\rangle,$ that is, the span of the
	elements of the form $ab$ where $a\in \mathcal{A} \setminus \{1\}$ and
	$b \in \mathbb{C}\langle x_1, \ldots, x_d\rangle.$
	
	To show that
		$$\sum_{\lambda\in \Lambda}
		u_{\lambda}(X)u_{\lambda}(X)^* <1,$$
	we will show by induction that
		\beq \label{inductivehypothesis}		
		\sum_{\lambda\in \Lambda, \text{deg }u_{\lambda}\leq n}
		u_{\lambda}(X)u_{\lambda}(X)^*
		+ \sum^{k_n}_{i=1} w_{i,n}(X)w_{i,n}(X)^* \leq \sum^{d}_{i=1} X_iX_i^*
		\eeq
	where $w_{i,n}$ form an orthonormal basis for $(\mathcal I \cap [H^2_d]_n)^{\perp}.$
	
	Note that for $n=1,$ Equation \eqref{inductivehypothesis} becomes
		$$\sum_{\lambda\in \Lambda, \text{deg }u_{\lambda} = 1}
		u_{\lambda}(X)u_{\lambda}(X)^*
		+ \sum^{k_1}_{i=1} w_{i,1}(X)w_{i,1}(X)^* = \sum^{d}_{i=1} X_iX_i^*$$
	which holds by Lemma \ref{orthosumsame} since the set of $u_\lambda$
	of degree one combined with the set of $w_{i,1}$ must form
	a basis for $[H^2_d]_1.$
	
	Now suppose that Equation \eqref{inductivehypothesis}
	holds for $n.$ That is,
	$$\sum_{\lambda\in \Lambda, \text{deg }u_{\lambda}\leq n}
		u_{\lambda}(X)u_{\lambda}(X)^*
		+ \sum^{k_n}_{i=1} w_{i,n}(X)w_{i,n}(X)^* \leq \sum^{d}_{i=1} X_iX_i^*.$$
	We will now show that it holds for $n+1.$
	Since
		$$\sum^{d}_{i=1} X_iX_i^* <1,$$
	we get that 
		$$\sum_{\begin{smallmatrix}\lambda\in \Lambda, \\
		 \text{deg }u_{\lambda}\leq n
		\end{smallmatrix}}
		u_{\lambda}(X)u_{\lambda}(X)^*
		+ \sum^{k_n}_{i=1} w_{i,n}(X)(\sum^{d}_{j=1} X_jX_j^*)w_{i,n}(X)^* \leq \sum^{d}_{i=1} X_iX_i^*.$$
	So
		$$\sum_{\lambda\in \Lambda, \text{deg }u_{\lambda}\leq n}
		u_{\lambda}(X)u_{\lambda}(X)^*
		+ \sum^{k_n,d}_{i=1, j=1} w_{i,n}(X)X_jX_j^*w_{i,n}(X)^* \leq \sum^{d}_{i=1} X_iX_i^*.
		$$
	Note that any $u_\lambda$ of degree $n+1$ must be in
	the subspace $(\mathcal I \cap [H^2_d]_n)^{\perp}\otimes [H^2_d]_1$
	by the definition of superorthogonality.
	Furthermore,
	the combination of the $u_\lambda$ of degree $n+1$ and $w_{i,n+1}$ form an
	orthonormal basis for $(\mathcal I \cap [H^2_d]_n)^{\perp}\otimes [H^2_d]_1.$
	On the other hand, the $w_{i,n}x_j$ form an orthonormal basis for
	$(\mathcal I \cap [H^2_d]_n)^{\perp}\otimes [H^2_d]_1.$
	So, by Lemma \ref{orthosumsame}, we get that
			\begin{flalign*}
		\sum^{k_n,d}_{i=1, j=1} w_{i,n}(X)X_jX_j^*w_{i,n}(X)^* = & & &
		\end{flalign*}
		$$
		\sum_{\begin{smallmatrix}\lambda\in \Lambda, \\
		\text{deg }u_{\lambda}= n+1\end{smallmatrix}}
		u_{\lambda}(X)u_{\lambda}(X)^* +
		\sum^{k_{n+1}}_{i=1} w_{i,n+1}(X)w_{i,n+1}(X)^*, 
		$$
	which immediately implies that
	$$
		\sum_{\lambda\in \Lambda, \text{deg }u_{\lambda}\leq n+1}
		u_{\lambda}(X)u_{\lambda}(X)^*
		+ \sum^{k_{n+1}}_{i=1} w_{i,n+1}(X)w_{i,n+1}(X)^* \leq \sum^{d}_{i=1} X_iX_i^*.$$
	
	\end{proof}


	\begin{lemma}\label{ball}
	Let $\mathcal A$ be a superorthogonal algebra.
	Let $p \in \mathcal{A}$. Then there exists a unique free polynomial
	$\hat{p}$ such that
		$$ \hat{p} \circ \Phi = p,$$
	where $\Phi$ is a superorthonormal basis map for $\mathcal A$ as in Lemma \ref{basismap}.
	Furthermore,
	$$\sup_{X \in \mathcal{C}^{d}} \|p(X)\| =
		\sup_{U \in \mathcal{C}^{\Lambda}} \|\hat{p}(U)\|.$$
	\end{lemma}

	\begin{proof}
	The existence of $\hat{p}$ follows from the fact that the coordinates of
	$\Phi$ form a basis for $\mathcal{A}.$
	The uniqueness of $\hat{p}$ follows from the fact that $\mathcal{A}$ is free
	by Lemma \ref{sofree}.
	
	So it remains to show the equality
		$$\sup_{X \in \mathcal{C}^{d}} \|p(X)\| =
		\sup_{U \in \mathcal{C}^{\Lambda}} \|\hat{p}(U)\|.$$
	
	To	show that
		$$\sup_{X \in \mathcal{C}^{d}} \|p(X)\| \geq
		\sup_{U \in \mathcal{C}^{\Lambda}} \|\hat{p}(U)\|,$$
	we show the equivalent inequality
		$$\sup_{X \in \cc{\mathcal{C}^{d}}} \|p(X)\| \geq
		\sup_{U \in \mathcal{C}^{\Lambda}} \|\hat{p}(U)\|.$$
	
	Similarly to our example for even functions (Section \ref{evensection}),
	given a $U \in \mathcal{C}^{\Lambda}$
	we would like to find an $X\in \cc{\mathcal{C}^{d}}$
	such that there is some projection $P$ so that
		\beq \label{equationprojection}  P\hat{p}({\Phi(X)})P  = \hat{p}(U)\eeq
	and thus that
		$$ \|p(X)\| \geq \|\hat{p}(U)\|.$$
	
	Let $$M_x = (M_{x_1},\ldots ,M_{x_d})$$
	as in Definition \ref{hardydef}.
	Note that $M_x \in \cc{\mathcal{C}^d}.$
	So, $\Phi(M_x) = (M_{u_\lambda})_{\lambda\in \Lambda}.$
	Decompose $H^2_d = \cc{\mathcal A}^{H^2_d} \oplus \mathcal J.$
	Since the algebra $\mathcal A$ is a joint invariant subspace for all the
	$M_{u_\lambda},$ we get that
	$$M_{u_\lambda} = \bpm M_{x_{\lambda}} & J_\lambda \\ 0 & K_\lambda \epm$$
	with respect to the decomposition,  where $J_\lambda$ and $K_\lambda$ are some operators
	which will be irrelevant to this discussion.
	The Frazho-Popescu dilation theorem \cite{frazho84, po89}
	states that for any $U\in \mathcal{C}^{\Lambda}$,
	there is a projection $\tilde{P}$ such
	that for any free polynomial $q,$
		$$\tilde{P}q((M_{x_{\lambda}} \otimes I))\tilde{P} = q(U).$$
	Thus, there is indeed an $X$ and a projection as desired in  \eqref{equationprojection},
	namely, taking $X = (M_{x_1}\otimes I,\ldots ,M_{x_d}\otimes I)$ and the projection
	as constructed above.
	
	Thus, for every $U \in \mathcal{C}^{\Lambda}$
	there exists an $X \in \cc{\mathcal{C}^d}$
	such that
		$$\|p(X)\| \geq\|\hat{p}(U)\|$$
	and so
		$$\sup_{X \in \cc{\mathcal{C}^{d}}} \|p(X)\| \geq
		\sup_{U \in \mathcal{C}^{\Lambda}} \|\hat{p}(U)\|.$$
		
	To see that
		$$\sup_{X \in \mathcal{C}^d} \|p(X)\| \leq \sup_{U \in \mathcal{C}^\Lambda} \|\hat{p}(U)\|$$
	we note that $\Phi(X) \in \mathcal{C}^\Lambda$ by Lemma \ref{basismap}.
	\end{proof}
	
	Now Theorem \ref{mainresultso} follows immediately.

\section{Example: Free functions in three variables invariant under the natural action of the cyclic group with three elements}\label{cyclic3example}
We now show how to construct a superorthonormal
basis for the ring of free polynomials in three variables which are invariant
under the natural action of the free group. That is, we want to understand 
free functions which satisfy the identity
	$$f(X_1,X_2,X_3) = f(X_2,X_3, X_1)$$
and show that they form a superorthogonal algebra.

Let $\sigma$ denote a generator of the cyclic group with three elements.
Define the action of $\sigma$ on free functions in three variables by
$$(\sigma \cdot f)(X_1,X_2,X_3) = f(X_2,X_3, X_1).$$
With this notation we are trying to understand functions such that
	$$\sigma \cdot f = f.$$

Let $\omega$ be a nontrivial third root of unity.
Consider the following three linear polynomials:
	\begin{align*}
	u_0(X_1,X_2,X_3) & = \frac{X_1 + X_2 + X_3}{\sqrt{3}} \\
	u_1(X_1,X_2,X_3) & = \frac{X_1 + \omega X_2 + \cc{\omega} X_3}{\sqrt{3}} \\
	u_{-1}(X_1,X_2,X_3) & = \frac{X_1 + \cc{\omega} X_2 +  \omega X_3}{\sqrt{3}}.
	\end{align*}
Clearly, the function $u_0$ is fixed by the natural action of the cyclic group with three elements.
However, $$\sigma \cdot u_1 = \omega u_1,$$ and $$\sigma \cdot u_{-1} = \cc{\omega} u_1.$$
That is, they are \emph{eigenfunctions} of the action $\sigma$ on free polynomials in three variables.

In fact, any product of the $u_i$ will be an eigenfunction of the action of $\sigma.$
For example,
		$$\sigma \cdot (u_{-1}u_1u_1) = \omega u_{-1}u_1u_1.$$
Thus, it can be observed that $\prod_j u_{i_j}$ is in the ring of invariant free polynomials under the action of the cyclic group
with three elements
if and only if $\sum_j i_j \equiv_3 0,$
and furthermore that products satisfying this condition span the algebra of free polynomials
in three variables which are fixed by the natural action of the cyclic group with three elements.

So, if we choose products $\prod^N_{j=1} u_{i_j}$ such that $\sum_j i_j \equiv_3 0,$ which are primitive in that no partial product $\prod^n_{j=1} u_{i_j}$ is in the ring of invariant free polynomials, we will obtain a basis for our algebra. To show that this basis is superorthogonal, it is
enough to show that any two distinct products $\prod^N_{j=1} u_{i_j},$ $\prod^N_{j=1} u_{k_j}$ are orthogonal.
However, by the grading lemma,
	$$\ip{\prod^N_{j=1} u_{i_j}}{\prod^N_{j=1} u_{k_j}} = \prod^N_j \ip{u_{i_j}}{u_{k_j}}.$$
Since the two products were assumed to be not equal, the orthogonality of the $u_i$ implies that at least
one of the $\ip{u_{i_j}}{u_{k_j}} =0,$ so we are done.

For the action of a general finite group,
an explicit construction of a superorthogonal basis for the ring of invariants is more difficult.
However, the existence of such a basis can be established using some basic representation theory which we
do in the next section. Later, we will return to the issue of an explicit construction for specific classes of groups for which the problem is tractible.

\section{The ring of invariant free polynomials is superorthogonal}
In order to prove Theorem \ref{mainresult} by Theorem \ref{mainresultso}, it is sufficient to show that the ring of invariant free polynomials is superorthogonal.
	\begin{definition}
		Let $\pi$ be a unitary representation of a finite group $G.$
		Let $\pi_n$ denote the action of $\pi$ on $[H^2_d]_n.$
		We make the identification that $\pi_1 = \pi.$
	\end{definition}
	
Let $p \in [H^2_d]_n$ and $q \in [H^2_d]_m$. Note that 
	$$\pi_{n+m} \cdot pq = (\pi_n\cdot p)(\pi_m \cdot q),$$
which translates to the following formal observation by the grading lemma.
	\begin{observation}\label{obs}
		Let $\pi$ be a unitary representation of a finite group $G.$
		Then,
		$$\flattensor{\pi_n}{\pi_m} = \pi_{n+m}$$
		with the identification made in the grading lemma.
	\end{observation}
Observation \ref{obs} implies that the action on homogeneous polynomials of degree $n$ is determined by the action on homogeneous polynomials of degree $1$.
	\begin{corollary}\label{tensorpowerrep}
		Let $\pi$ be a unitary representation of a finite group $G.$
		Then,
		$$\pi_n = \pi^{\otimes n}$$
		with the identification made in the grading lemma.
	\end{corollary}

We can now prove that the ring of invariant free polynomials is superorthogonal by
constructing a basis for it. 

\begin{theorem}\label{freeringortho}
Let $G$ be a finite group. Let $\pi: G \to U_d$ be a group representation. The ring of invariant free polynomials for $\pi$ is superorthogonal. 
\end{theorem}
\begin{proof}
The action of the group on homogeneous polynomials of degree $i$ is given by $\pi^{\otimes i}(g) = \pi_i(g)$. To show that the ring of invariant free polynomials is superorthogonal, we construct a superorthogonal basis recursively.

	Let $V_1 = [H^2_d]_1$, and if $i>1$, let
	$$V_{i} = \tilde{V}_{i-1}^\perp \otimes [H^2_d]_1,$$
	where $\tilde{V}_{i-1}$ is defined as follows. For each $i$, consider
	$\tilde{\pi}_i = \pi_i|_{V_i}$
and decompose 
	$V_i = \tilde{V_i} \oplus \tilde{V_i}^\perp,$
where $\tilde{V_i}$ is the space which is fixed by the action of $\tilde{\pi_i}.$

Let $(u_{i,k})$ be an orthonormal basis of $V_i$. We will show that for all $p,q$ in the ring of invariant free polynomials,
	$$u_{i,k} p \perp u_{j,l}q$$
for $j \leq i$, i.e. that $u_{i,k}$ and $u_{j,l}$ are superorthogonal.

When $i = j$, for $l\neq k$, by the grading lemma,
	\begin{align*}
		\ip{u_{i,k}p}{u_{j,l}q} &= \ip{u_{i,k}}{u_{j,l}}\ip{p}{q} \\
		&= 0.
	\end{align*}
When $j < i$, 
by the recursive construction, $u_{i,k} \in \tilde{V_j}^\perp \otimes [H^2_d]_1^{\otimes a}$ and $u_{j,l} \in \tilde{V_j} \otimes [H^2_d]_1^{\otimes b}$ for some $a, b$,
and therefore
	$$\ip{u_{i,k}p}{u_{j,l}q} = 0.$$
	
To show that our recursively generated sequence $(u_{i,k})$ is a basis, consider the following. If $p \in [H^2_d]_n$ is in the ring of invariant free polynomials, note that the projection of $p$ onto any $u_{i,k}[H^2_d]_{n-i}$ is in the ring of invariant free polynomials. 	
So for any $p$ in the ring of invariant free polynomials, 
	$$p = \sum u_{i,k}q_{i,k} + r,$$
where $q_{i,k}$ is in the ring of invariant free polynomials, and $r \in \tilde{V_n}^\perp$ and in the ring of invariants.
So, by construction, $r = 0$.
\end{proof}

\section{Structure of the superorthogonal basis for the ring of invariants}


\subsection{Counting}\label{counting}
We now calculate the number of elements in any superorthonormal basis
for the ring of invariant free polynomials of a given degree in terms of generating functions.
The main result of this section is as follows.
\begin{theorem} \label{countingTheorem}
	Let $\pi$ be a unitary representation of a group $G$ and $\chi= \tr{\pi}$ be the character corresponding to $\pi.$
	Let $\mathcal{C}_G$ be a set of representatives for the conjugacy classes of $G.$
	Let $g_{n}$ be the number of free polynomials of degree $n$
	in some superorthogonal basis for the ring of invariant free polynomials
	of $\pi.$
	Then,
	$$ g(z) =\sum_{n=0}^{\infty}g_{n}z^{n} = 1-|G|\frac{\prod_{\sigma \in \mathcal C_G}
	(1-\chi(\sigma)z)}{\sum_{\tau \in \mathcal C_G} \#C_\tau\prod_{\sigma \neq \tau}(1-\chi(\sigma)z)}$$
	where $\#C_\tau$ denotes the number of elements in the conjugacy class of $\tau.$
	
	Namely, the number of generators of a given degree is independent of the choice of superorthogonal basis.
\end{theorem}
We prove Theorem \ref{countingTheorem} in Section \ref{countingTheoremProof}.

For example, consider
symmetric functions in three variables.
That is, take the group $S_3$ acting on $\mathbb{C}^3$ via the representation $\pi(\sigma)e_i = e_{\sigma(i)}.$
According to Theorem \ref{countingTheorem}, the necessary information can be conveniently compiled in the following table.
\begin{center}
\begin{tabular}{ c | c | c | c}
       & $e$ & (1 2) & (1 2 3) \\ \hline
  $\chi$ & 3 & 1 & 0 \\
  $\#C_\sigma$ & 1 & 3 & 2 \\
\end{tabular}
\end{center}
So,
	\begin{align*} g(z)& = 1- 6\frac{(1-3z)(1-z)(1 - 0z)}
	{(1-z)(1-0z)+ 3(1-3z)(1-0z)+ 2(1-z)(1 - 3z)}, \\
	&=\frac{z-2z^{2}}{1-3z+z^2}, \\
	&=z+z^2+2z^2+5z^3+13z^4+..., \\
	&=z+\sum_{n=2}^{\infty}F_{2n-3}z^{n},
	\end{align*}
where $F_n$ denotes the $n$-th Fibonacci number.

Thus, in general, with the help of computer algebra software, it is a simple exercise to calculate the number of free polynomials of degree $n$
	in some superorthogonal basis for the ring of invariant free polynomials
	of $\pi.$

\subsubsection{The proof of Theorem \ref{countingTheorem}}\label{countingTheoremProof}
Let $f_{n}$ be the dimension of invariant homogeneous free polynomials of degree $n.$
Let $g_{n}$ be the number of free polynomials of degree $n$ in some superorthogonal basis for the ring of invariant free polynomials
of $\pi.$
Let
	$$
		g(z) =\sum_{n=0}^{\infty}g_{n}z^{n}, 
		f(z) =\sum_{n=0}^{\infty}f_{n}z^{n}. 
	$$
\begin{lemma} \label{gfrelation}
		$$g(z)=\frac{f(z)-1}{f(z)}.$$
\end{lemma}
\begin{proof}
	It can be shown using enumerative combinatorics that 
		$$f_n = \sum_{i_1+\ldots+i_k =n} \prod^k_{1} g_{i_j}.$$

	Thus,
	\begin{align*}\frac{1}{1-g(z)} &=\frac{1}{1-\sum_{n=0}^{\infty}g_{n}z^{n}} \\
	&=\sum_{m=0}^{\infty}\left(\sum_{n=0}^{\infty}g_{n}z^{n}\right)^{m} \\
	&=\sum_{l=0}^{\infty}\sum_{i_1+\ldots+i_k =n} \prod^k_{1} g_{i_j} z^{l} \\
	&=\sum_{n=0}^{\infty}f_{n}z^{n} \\&=f(z).\end{align*}
	
	Calculating $g$ from $f$ gives that
		$$g(z)=\frac{f(z)-1}{f(z)}.$$
\end{proof}


In order to calculate $f_n$, we use character theory (see Serre \cite[Chapter 2]{serreRep}).

Let $\rho$ be a representation for $G$. For each $\sigma\in G$, put
$\chi_{\rho}(g)=\tr (\rho(\sigma)).$ $\chi_{\rho}$ is the \dfn{character} of $\rho.$ 
Let $\phi_{1}, \phi_{2}$ be functions from $G$ into $\mathbb{C}.$ The scalar product $\ip{\cdot}{\cdot}$ is defined by
$$\ip{\phi_{1}}{\phi_{2}}=\frac{1}{|G|}\sum_{\sigma\in G} \phi_{1}(\sigma)\overline{\phi_{2}(\sigma)}.$$
Let $\tau:G\rightarrow\mathcal{U}_{1}$ be the trivial representation. 
The dimension of the space of fixed vectors of the action of $\rho$ is given by $\ip{\chi_{\rho}}{\tau}.$
(See Serre \cite[Chapter 2]{serreRep})

We make the identification $\chi = \chi_{\pi}.$

\black
\begin{lemma}\label{fgf}
	$$f(z) =\frac{1}{|G|}\sum_{\sigma\in G}\frac{1}{1-\chi(\sigma)z}.$$
\end{lemma}
\begin{proof}
	By Lemma \ref{tensorpowerrep}, the action of $\pi$ on homogenous polynomials is given by $\pi^{\otimes n}$. 
	\begin{align*}
	f(z) & =\sum_{n=0}^{\infty}f_{n}z^{n}
	\\
	& =\sum_{n=0}^{\infty}\ip{\chi^{n}}{\tau}z^{n}
	\\
	& =\frac{1}{|G|}\sum_{n=0}^{\infty}\sum_{\sigma\in G}\chi(\sigma)^{n}\overline{\tau(\sigma)}z^{n}
	\\
	& =\frac{1}{|G|}\sum_{n=0}^{\infty}\sum_{\sigma\in G}(\chi(\sigma)z)^{n}
	\\
	& =\frac{1}{|G|}\sum_{\sigma\in G}(1-\chi(\sigma)z)^{-1}. \\
	& =\frac{1}{|G|}\sum_{\sigma\in \mathcal{C}_G}\frac{\#C_\sigma}{1-\chi(\sigma)z}.
	\end{align*}
\end{proof}
Now Theorem \ref{countingTheorem} follows as an immediate corollary of Lemma \ref{gfrelation} and Lemma \ref{fgf}.

\subsection{The ring of invariants of a finite abelian group}\label{abelian}


\begin{theorem}\label{abelianbasistheorem}
Let $G$ be a locally compact abelian group, and let $\pi:G \to U_d$ be a unitary representation of $G$. Let $v_1, \ldots, v_d$ be an orthonormal set of vectors of $[H^2_d]_1$ with corresponding characters $\chi_1, \ldots, \chi_d$ such that $\pi(g)v_i = \chi_i(g) v_i$. A basis for the free ring of invariants as a vector space is given by
	$$B = \left\{ v^J | \chi^J = \tau \right\}.$$
Furthermore,
	$$\tilde{B} = \left\{ v^I \in B | \forall v^J, v^K \neq 1 \in B : v^I \neq v^Jv^K\right\}$$
forms a superorthonormal basis.
\end{theorem}
\begin{proof}
	Note that $\pi^n v^J = (\chi v)^J = \chi^J v^J.$
	So,
		$$\frac{1}{|G|}\sum_{g\in G} \chi^J(g) v^J = \ip{\chi^J}{\tau}v^J.$$
	So, any invariant free polynomial is in the span of the $v^J\in B,$
	which are themselves invariant homogeneous free polynomials, since otherwise 
	$\ip{\chi^J}{ \tau} = 0$ by the orthogonality of characters.
	
	To show that $\tilde{B}$ is superorthogonal, note
	for any two distinct words,
		$$\ip{v^K}{v^L} =  0$$
	by the grading lemma, which implies the claim.

\end{proof}
So, the calculation of a superorthogonal basis for the ring of invariants of a finite abelian group
is tractable in general.

The superorthogonal basis given in Theorem \ref{abelianbasistheorem} also implies that off of
some variety, the ring of invariant functions is finitely generated. So,
there are finitely many invariant free rational functions such that any invariant free polynomial can be written in terms of them, in
spite of the fact that the ring of invariant free polynomials
is not itself finitely generated.
Similar phenomena occur in real algebraic
geometry, such as the fact that 
positive polynomials cannot
be written as sums of squares of polynomials \cite{motzkin}, 
but can be written as sums of squares of rational functions, i.e.
Artin's resolution of Hilbert's seventeenth problem \cite{artin}.
\begin{theorem}
	With the notation of Theorem \ref{abelianbasistheorem} the ring of invariant functions in the algebra generated by
	$ v_1,\ldots, v_d, v_1^{-1}, \ldots, v_d^{-1}$ is finitely generated.
\end{theorem}
\begin{proof}
We note that
by Pontryigan duality theorem \cite[Theorem 1.7.2]{rudin62},
the characters of an abelian group $G$ form a group $\hat{G}$ under multiplication
which is noncanonically isomorphic to $G.$
Let $F_d$ be the free group with $d$ generators.
Let $H = \{I \in F_d | \chi^I =0\}.$
Consider the short exact sequence
	$$0 \rightarrow H \rightarrow F_d \rightarrow \hat{G} \rightarrow 0.$$
So, $H$ is of finite index in $F_d$ and is thus finitely generated, which then
implies that the ring
is finitely generated.
\end{proof}

\subsection{Example: Ad hoc methods for symmetric functions in three variables}

	We now turn our attention to symmetric functions in three variables.

	Let  $u_{0},u_{1},u_{-1}$ be as in Section \ref{cyclic3example} and let  $(b_{n})_{n\in\mathbb{N}}$ be the constructed superorthogonal basis for cyclic free polynomials in $3$ variables, where  we fix $b_0 = u_0.$
	
	Let $\tau$ be the following action on a free function $f$: 
		$$\tau\cdot f(X_{1},X_{2},X_{3})=f(X_{1},X_{3},X_{2}).$$
	Note that
		$$\tau\cdot u_{0}=u_{0},$$   
		$$\tau\cdot u_{1}=u_{-1},$$
		$$\tau\cdot u_{-1}=u_{1}.$$
We recall that each $b_n$ 
	is of the form $$b_n = \prod_k u_{i_k}$$ such that
	$\sum_k i_k \equiv_3 0.$	
	So, $\tau \cdot b_n = b_{\tilde{n}}$
	for some $\tilde{n},$ since $\tau \cdot \prod_k u_{i_k} = \prod_k u_{-i_k},$
	and so $\sum_k -i_k \equiv_3 0.$

	For $n >0$ define
		$$c_{n}^{0}=\frac{b_{n}+\tau\cdot b_{n}}{\sqrt 2},$$
		$$c_{n}^{1}=\frac{b_{n}-\tau\cdot b_{n}}{\sqrt 2}.$$
	For $n =0,$ let $c^0_0 = u_0.$
	
	Note that each $c_{n}^{0}$ is a symmetric free polynomial and that the product of an even number of $c_{n}^{1}$ is also symmetric. In fact,
		$$\tau \cdot  c_{n}^{0} = c_{n}^{0},$$
		and
		$$\tau \cdot  c_{n}^{1} = -c_{n}^{1}.$$
	
	Set 
	$$B=\left\{\Pi_{i=1}^{N}c_{i}^{k_{i}}|
	\sum_{i=1}^{N}k_{i}\equiv_2 0, \forall n<N, \sum_{i=1}^{n}k_{i} \not\equiv_2 0\right\}.$$

	Now,
	$B$ is a superorthogonal basis for the symmetric free polynomials in 3 variables.	
	The elements of $B$ of degree 4 and less are given in the following table.
	\begin{center}
	\begin{tabular}{ c | c}
    			degree & basis elements \\ \hline
 		 	1 & $u_{0}$ \\
  			2 & $\frac{u_{1}u_{-1}+u_{-1}u_{1}}{\sqrt 2}$ \\
  			3 & $\frac{u_{1}^{3}+u_{-1}^{3}}{\sqrt 2},
  			\frac{u_{1}u_{0}u_{-1}+u_{-1}u_{0}u_{1}}{\sqrt 2}$ \\
  			4 & $\frac{u_{1}^{2}u_{-1}^{2}+u_{-1}^{2}u_{1}^{2}}{\sqrt 2},
  			\frac{u_{1}u_{0}u_{1}^{2}+u_{-1}u_{0}u_{-1}^{2}}{\sqrt 2},
  			\frac{u_{1}u_{0}^2u_{-1}+u_{-1}u_{0}^2u_{1}}{\sqrt 2},
  			$ \\
  			 &
  			$\frac{u_{1}^{2}u_{0}u_{1}+u_{-1}^{2}u_{0}u_{-1}}{\sqrt 2},
  			(\frac{u_{1}u_{-1}-u_{-1}u_{1}}{\sqrt 2})^{2}$
	\end{tabular}
	\end{center}
	Note the table agrees with the generating function obtained earlier.
	As there are $13$ elements of the basis of degree $5,$ we will stop here.
	
We remark that the method above of iteratively constructing can be applied, in principal, for any solvable group,
since we exploited the fact that $\mathbb{Z}_2\cong S_{3}/\mathbb{Z}_{3}$ acts on the ring of free polynomials invariant under $\mathbb{Z}_3.$ 
Since the ring of free polynomials invariant under $\mathbb{Z}_3$ is itself
isomorphic to a free algebra
in infinitely many variables, we are essentially repeating the construction done for
any abelian group.



\printindex
\bibliography{references}
\bibliographystyle{plain}

\end{document}